\documentclass[a4paper,12pt,reqno]{amsart}

\usepackage{bigints}
\usepackage{ esint }
\usepackage{amsmath}
\usepackage{amsfonts}
\usepackage{amssymb}
\usepackage{amsthm}
\usepackage{mathrsfs}
\usepackage[shortlabels]{enumitem}
\usepackage{graphicx}
\usepackage{color}
\usepackage[latin1]{inputenc}
\usepackage{MnSymbol}
\usepackage{enumitem}
\usepackage{extpfeil}
\usepackage{mathtools}
\usepackage{url}
\usepackage[margin=1.1in]{geometry}
\usepackage{fancyhdr}
\usepackage[all,cmtip]{xy}
\usepackage{array,multirow}
\usepackage[x11names]{xcolor}
\usepackage{tikz}
\usepackage{pgf}
\usepackage{appendix}
\usepackage{float}
\RequirePackage{hyperref}

\theoremstyle{plain}
\newtheorem{theorem}{Theorem}[section]

\newtheorem{theoremletter}{Theorem}

\newtheorem{lem}[theorem]{Lemma}
\newtheorem{prop}[theorem]{Proposition}
\newtheorem{teo}[theorem]{Theorem}
\newtheorem{defi}[theorem]{Definition}
\newtheorem{claim}{Claim}
\newtheorem*{Convention}{Convention}
\newtheorem{cor}{Corollary}


\newcommand{\Ga}{\Gamma}

\newcommand{\ga}{\gamma}
\newcommand{\eps}{\epsilon}

\newcommand{\supp}{\textrm{supp}}

\newcommand{\df}{\mathcal F} 
\newcommand{\La}{\Lambda}
\newcommand{\la}{\lambda}
\newcommand{\Om}{\Omega}

\newcommand{\tp}{\otimes}
\newcommand{\diam}{\text{diam}}

\newcommand{\act}{\curvearrowright}

\title{Isoperimetric Profiles and Regular Embeddings of locally compact groups}

\author{Juan Paucar}
\address{Institut de Math\'ematiques de Jussieu-Paris Rive Gauche, Universit\'e Paris Cit\'e, 75205 Paris Cedex 13, France}
\email{juan.paucar@imj-prg.fr}

\date{\today}

\begin{document}
\begin{abstract}
In this article we extend the notion of $L^p$-measure subgroups couplings, a quantitative asymmetric version of measure equivalence  that was introduced by Delabie, Koivisto, Le Ma\^itre and Tessera for finitely generated groups, to the setting of locally compact compactly generated unimodular groups. As an example of these couplings; using ideas from Bader and Rosendal, we prove a "dynamical criteria" for the existence of regular embeddings between amenable locally compact compactly generated unimodular groups, namely the existence of an $L^\infty$-measure subgroup coupling that is coarsely $m$-to-$1$. We also prove that the existence of an $L^p$-measure subgroup that is coarsely $m$-to-$1$ implies the monotonicity of the $L^p$-isoperimetric profile, as well as sublinear version of this result. As a corollary we obtain that the $L^p$-isoperimetric profile is monotonous under regular embeddings, as well as coarse embeddings, between amenable unimodular locally compact compactly generated groups.  
 
 \vspace*{2mm} \noindent{2020 Mathematics Subject Classification:} 37A20,  22D05, 
 22F10, 22D40, 20F69, 20F65.

\end{abstract}
\maketitle

\tableofcontents

\section{Introduction}

\subsection{Regular Embeddings and Measure Subgroup Couplings}
Let's consider $\Ga$ and $\La$ two finitely generated groups with fixed word metrics. In his seminal article \cite{Gromov}, Gromov proved the following characterization of the existence of quasi-isometries between finitely generated groups:

\begin{prop}[Section $0.2.C'_2$ in \cite{Gromov}] 
The following are equivalent:
\begin{enumerate}
    \item $\Ga$ and $\La$ are quasi-isometric.
    \item There exists a locally compact topological space $\Om$ with a continuous $\Ga$-action and a continuous $\La$-action that commute with each other, such that both actions are proper and cocompact. 
\end{enumerate}
\label{DynCriteriaGromov}    
\end{prop}
We will say that such a topological space $\Om$ is a \textbf{topological equivalence coupling} between $\Ga$ and $\La$.
This proposition motivated Gromov to introduce the following definition:
\begin{defi}[Definition $0.5.E_1$ in \cite{Gromov}]
    We say that a standard measured space $(\Om,\mu)$ is a \textbf{ measure equivalence coupling} between $\Ga$ and $\La$ if there exists measure preserving commuting $\Ga,\La$-actions such that both actions are free and have fundamental domains of finite measure.  
    We say as well that $\Ga$ and $\La$ are \textbf{measure equivalent} if there exists a measure equivalence coupling between them.
\end{defi}
This marked the starting point of Measured Group Theory. For a comprehensive survey of this vast subject and the links with other areas, see the references \cite{Gaboriau} and \cite{Furman}. 

Furthermore, Proposition \ref{DynCriteriaGromov} has also an asymmetric version as well, to be more precise, let's first define:
\begin{defi}
    We say that $f\colon \Ga\to \La $ is a \textbf{coarse embedding} from $\Ga$ to $\La$ if $f$ is $L$-Lipschitz for some $L>0$ and for any $R>0$ there exists $C(R)>0$ such that:
    \begin{equation*}
        \diam_\Ga(f^{-1}(B_\La(\la,R)))\leq C(R)
    \end{equation*}
    for all $\la$ in $\La$.
\end{defi}
Note first, that this definition is equivalent to the definition given in \cite{Shalom} of uniform embeddings, moreover we have the following characterization of the existence of coarse embeddings between finitely generated groups: 
\begin{prop}\cite[Theorem 2.1.2]{Shalom}
The following are equivalent:
\begin{enumerate}
    \item There exists a coarse embedding from $\Ga$ to $\La$.
    \item There exists a locally compact topological space $\Om$ with a continuous $\Ga$-action and a continuous $\La$-action that commute with each other, such that both actions are proper actions, and the $\La$-action is cocompact.
\end{enumerate}
\label{CoarseEmbedding}
\end{prop}

We will say then that such a topological space $\Om$ will be a \textbf{Topological Subgroup coupling} from $\Ga$ to $\La$. See \cite{Sauer} for other consequences of this characterization.

It is then natural to ask ourselves if, following the spirit of Measured Group Theory, there is a measured analog of this proposition. This turns out to be the case, as seen in \cite{Romain}, in this article the authors began a study of quantitative asymmetric versions of measure equivalence by demanding less conditions on the actions over the coupling space, which they call measure subgroup couplings and measure subquotient couplings. Moreover, they also study quantitative versions of this concept, which includes, for example $L^p$-Measure Equivalence and $L^p$-measure subgroups. They also consider regular embeddings; a notion introduced in \cite{Benjamini} for graphs and studied for non-discrete spaces in \cite{Poincare}, as the measured analogue of coarse embeddings.
\begin{defi}
    We say that $f  \colon  \Ga\to \La$ is a \textbf{regular embedding} from $\Ga$ to $\La$ if $f$ is $L$-Lipschitz for some $L>0$ and there exists $m>0$ such that for any $\la$ in $\La$:
    \begin{equation*}
        \#\{ f^{-1}(\la)\}\leq m.
    \end{equation*}
\end{defi}
Note that this definition gives us uniform control of the preimages of balls in a measured sense and not in a metric sense as in the case of a coarse embedding. For example, the application $f \colon \mathbb Z\to \mathbb Z$ given by $f(n)=|n|$ is a regular embedding but not a coarse embedding.


In the Theorem 5.4 of \cite{Romain}, the authors proved the following theorem that can be considered as a measure theoretical version of Proposition \ref{CoarseEmbedding}:
\begin{teo}
    Given $\Ga$ amenable. The following are equivalent: 
    
    \begin{enumerate}
        \item There exists a regular embedding from $\Ga$ to $\La$.
        \item There exists $(\Om,\mu)$ a $L^\infty$-measure subgroup coupling from $\Ga$ to $\La$ that is $m$ to $1$ for some $m>0$.
    \end{enumerate}
    \label{RomainA}
\end{teo}

On the other hand, Bader and Rosendal proved in \cite{Rosendal}, the analog of Proposition \ref{DynCriteriaGromov} for locally compact compactly generated groups. Using their ideas as inspiration, we extend then the Theorem \ref{RomainA} to the locally compact setting, thus obtaining the first main contribution of this article:
\begin{theoremletter}[Theorem \ref{teoprincipal}]
    Given $H,G$ two locally compact compactly generated unimodular groups, with $H$ being amenable, we have the following equivalence:
    \begin{enumerate}
        \item There exists a regular embedding from $H$ to $G$.
        \item There exists a $L^\infty$-measure subgroup coupling from $H$ to $G$ that is coarsely $m$ to $1$ for some $m>0$.
    \end{enumerate}
    \label{A}
\end{theoremletter}

Note that the approach given in \cite{Romain} doesn't translate well to the locally compact setting, since they use that $\Ga$ can be well ordered. However, in the locally compact setting this is not possible to do in a Borel way. It is also not clear how to obtain an injective map out of a regular embedding between locally compact groups. This is where the ideas coming from \cite{Rosendal} come into play, note however that applying directly their ideas would only gives us a topological subgroup coupling out of a coarse embedding between locally compact groups. So in order to obtain results in the measured setting, that is for regular embeddings, some extra work is needed. The main difficulty will be the proof of the smoothness of the $H$-action(Claim \ref{Important}) that will use Lemma \ref{Effros} and Proposition \ref{regularitylemma} as its main ingredients. Note that such action is not proper, so the proof in Lemma \ref{properaction} that proper actions are smooth(see \cite[Example 6.11]{Koivisto} as well) can't be used.

\subsection{Monotonicity of the isoperimetric profile}

In the discrete setting, the isoperimetric profile of a group is an important invariant that quantifies the amenability of the group. To define this concept more precisely, let \( p \geq 1 \) be fixed, and let \( \Gamma \) be a finitely generated group with a fixed generating set \( S_\Gamma \). For any function \( f \in \ell^p(\Gamma) \), we define the \( L^p \)-norm of the left-gradient of  $f$ as:
\[
||\nabla^l_\Gamma f||_p^p = \sum_{s \in S_\Gamma}  \sum_{\ga\in \Ga} |f(sg)-f(g)|^p
\]
Given a function $f\colon \Ga\to \mathbb R$, its support is the set $\supp f=\{g\in\Gamma\colon f(g)\neq 0\}$.
We then define the \( \ell^p \)-isoperimetric profile of \( \Gamma \) as the non-decreasing function:

\[
j_{p,\Gamma}(n) = \sup_{|\text{supp}(f)| \leq n} \frac{||f||_p}{||\nabla_\Gamma f||_p}
\]

Here, we are interested in the asymptotic behavior of this function as \( n \to \infty \). To express this more rigorously, for two increasing real-valued functions \( f \) and \( g \), we say that \( f \) is asymptotically less than \( g \) and write \( f \preccurlyeq g \) if there exists a positive constant \( C \) such that \( f(n) = O(g(Cn)) \) as \( n \to \infty \). If both \( f \preccurlyeq g \) and \( g \preccurlyeq f \), we say that \( f \) and \( g \) are asymptotically equivalent and write \( f \approx g \). The asymptotic behavior of a function is its equivalence class modulo \( \approx \).

In the specific case of $p=1$, the $\ell^1$ isoperimetric profile 
is simply called the isoperimetric profile because of the following geometric interpretation \cite{Coulhon}:
\[j_{1,\Gamma}(n)\approx \sup_{|A|\leq n} \frac{|A|}{|\partial A|},\]
where $\partial A=S_\Gamma A\bigtriangleup A$.
Certain authors prefer to work with the  F\o lner 
function, defined for all $k\geq 1$ as follows
\[\text{F\o l}_\Gamma(k)= \inf\left\{|A|\colon  \frac{|\partial A|}{|A|}\leq 
1/k\right\}.\]
Note that the isoperimetric profile and the F\o lner 
function are generalized inverses of one another.


It is to be noted that the asymptotic behavior of \( j_{p,\Gamma}(n) \)  does not depend on the choice of the generating set \( S_\Gamma \). Additionally, \( j_{p,\Gamma} \) is unbounded if and only if \( \Gamma \) is amenable. So the isoperimetric profile can be interpreted as a measurement of amenability: the faster it tends to infinity, the "more amenable" the group is.
This is illustrated by the following examples. For all $p\geq 1$ we have:

\begin{itemize}
	\item $j_{p,\Gamma}(n)\approx n^{1/d}$ for a group of polynomial growth of degree $d$ (e.g.\ for $\Gamma=\mathbb Z^d$) (see  \cite{Coulhon});
	\item $j_{p,\Gamma}(n)\preccurlyeq \log n$ for all amenable groups with exponential growth \cite{coulhonIsoperimetriePourGroupes1993};
	\item $j_{p,\Gamma}(n)\approx \log n$ for a wide class of solvable groups with exponential growth including those that are polycyclic, or the so-called ``lamplighter'' groups, i.e.\ group of the form $F\wr \mathbb Z$, where $F$ is a finite non-trivial group \cite{coulhonGeometricApproachOndiagonal2001,pittetFolnerSequencesPolycyclic1995,pittetIsoperimetricProfileHomogeneous2000,IsoProfilLC};
	\item $j_{1,\Gamma}(n)\approx (\log n)^{1/d}$ for a group of the form $F\wr \Sigma$, where $F$ is a non-trivial finite group and $\Sigma$ has polynomial growth of degree $d$ \cite{erschlerIsoperimetricProfilesFinitely2003}.
\end{itemize}



Furthermore, in \cite[Theorem 4.3]{Romain} the authors proved the following monotonicity result:
\begin{teo}
\label{RomainB}
If there exists a $L^p$-measure subgroup coupling from $\Ga$ to $\La$ that is $m$-to-$1$, we then have that:
\begin{equation*}
    j_{p,\La}\preccurlyeq j_{p,\Ga} 
\end{equation*}
\end{teo}
As well as its sublinear version(cf. \cite[Theorem 4.4]{Romain}):
\begin{teo}
\label{RomainC}
Given $\varphi\colon (0,+\infty)\to (0,+\infty)$ such that $\varphi$ and $t\mapsto \frac{t}{\varphi(t)}$ are non-decreasing. If there exists a $\varphi$-measure subgroup coupling from $\Ga$ to $\La$ that is $m$-to-$1$, we then have that: 
\begin{equation*}
    \varphi \circ j_{1,\La} \preccurlyeq j_{1,\Ga}
\end{equation*}
\end{teo}

On the other hand, in \cite{IsoProfilLC}, the $L^p$-isoperimetric profile of unimodular amenable locally compact second countable compactly generated groups is introduced. It is natural then to ask about an analogue of Theorem \ref{RomainB} and Theorem \ref{RomainC} in the locally compact setting. We prove the desired analog statements, thus obtaining the second main contribution of the article. Let's consider in all this section $H,G$ two unimodular locally compact compactly generated groups, we then have that:
\begin{theoremletter}[Theorem \ref{teoLp}]
    \label{B}
Suppose there exists a $L^p$-measure subgroup coupling from $H$ to $G$, that is coarsely $m$-to-$1$, we then have that:
\begin{equation*}
    j_{p,G} \preccurlyeq j_{p,H}
\end{equation*}
\end{theoremletter}
As well as its sublinear version:
\begin{theoremletter}[Theorem \ref{teosublinear}]
\label{C}
    Let $\varphi\colon (0,\infty)\to (0,\infty)$ a function such that $\varphi$ and $t\mapsto \frac{t}{\varphi(t)}$ are non-decreasing functions. Suppose there exists a $\varphi$-measure subgroup coupling from $H$ to $G$, that is coarsely $m$-to-$1$, we then have that:
\begin{equation*}
\varphi \circ j_{1,G}  \preccurlyeq j_{1,H}  
\end{equation*}
\end{theoremletter}

It is now clear that we obtain the following corollary:
\begin{cor}
  If there exists a regular embedding from $H$ to $G$ and $H,G$ are amenable, then for any $p\geq 1$ we know that:
\begin{equation*}
   j_{p,G}\preccurlyeq j_{p,H} 
\end{equation*}
\end{cor}

On the other hand, another related coarse invariant is the return probability to the origin. More precisely, in the discrete setting for a group $\Ga$ with finite symmetric generating set $S$, we can consider the natural random walk on the Cayley graph of $(\Ga,S)$ starting at the identity, and denote by $p_{\Ga,S}^{2n}$ the probability that in $2n$ steps the random walk will visit the identity again. We will be interested on its asymptotic behaviour as $n$ goes to infinity. It was proven in \cite{Pittet} that such asymptotic behaviour doesn't depend on the choice of generating set and moreover it is invariant under quasi-isometries. Furthermore, in  the locally compact setting, that is for $G$ an unimodular locally compact second countable compactly generated group, the analog of such an invariant can be defined and also proven its invariance under change of generating sets (cf. \cite{Dretete}).
Moreover, as the isoperimetric profile, it is an invariant that measures how amenable $\Ga$ is, in the sense that $\Ga$ is not amenable if and only if $p_{\Ga}^{2n}\approx e^{-n}$. It is also known that the $L^2$-isoperimetric profile and the return probability to the origin are related under some mild assumptions, this was seen in the discrete setting in \cite{Coulhon} and in the locally compact setting in \cite{IsoProfilLC}. More precisely we have the following: 
\begin{prop}[Theorem 9.2 in \cite{IsoProfilLC}]
\label{proba}
    Given $G$ an unimodular locally compact second countable group with compact symmetric generating set $S$. Consider $\ga$ the function defined by:
    \begin{equation}
        t=\bigintssss_1^{\ga(t)} \frac{j_{2,G}(v)^2}{v} dv.\label{alpha}    \end{equation}
 If the logarithmic derivative of $\ga$ has at most polynomial growth, then we have that:
\begin{equation*}
    p_{G}^{2n} \approx \dfrac{1}{\ga(2n)}
\end{equation*}
\end{prop}

Note that Proposition \ref{proba} is proved in \cite{IsoProfilLC} in the more general context of metric measured spaces $(X,d,\mu)$ and of viewpoints $(P_x)_{x\in X}$; these are a family of appropiate probabilities absolutely continuous with respect to $\mu$ that depend on $x$. In particular, in the context of an unimodular locally compact second countable compactly generated group $G$, the metric to be considered is a left-invariant metric quasi-isometric to the word metric and the measure $\mu$ will be a fixed left-invariant Haar measure.

Moreover, the condition given in Proposition \ref{proba} is a mild condition in the sense that it is satisfied for a very large class of amenable groups (see Table 1 in \cite{BendikovPittetSauer}). Let's recall as well that such condition is called condition $(\delta)$ in \cite{Coulhon}, by an abuse of notation let's say that $G$ satisfies condition $(\delta)$ if $\gamma$ defined by equation (\ref{alpha}) satisfies such a condition.

Now as consequence of this theorem we obtain the following corollary:
\begin{cor}
Given $H,G$ amenable that satisfy the condition $(\delta)$ of Proposition \ref{proba}, if there exists a regular embedding from $H$ to $G$,  we then have that:
\begin{equation*}
    p_G^{2n}\preccurlyeq p_H^{2n}
\end{equation*}
\end{cor}

On the other hand, we also have the following:
\begin{cor}
Given $G$ amenable and unimodular, if there exists a regular embedding from $G$ to $GL(d,k)$ where $k$ is a finite product of local fields, we obtain that:
\begin{equation*}
    \log t \preccurlyeq j_{p,G}(t)
\end{equation*}
\end{cor}
\begin{proof}
    Since $T(d,k)$ the group of upper triangular matrices in $GL(d,k)$ is cocompact in $GL(d,k)$, we can suppose that there exists a regular embedding from $G$ to $T(d,k)$, however $T(d,k)$ is not unimodular. In order to obtain a regular embedding to an unimodular amenable group, let's note that $T(d,k)$ is a semidirect product of $A(d,k)$ and $N(d,k)$, where $A(d,k)$ is abelian and $N(d,k)$ is nilpotent. Let's denote by $\sigma\colon A(d,k)\to \text{Aut}(N(d,k))$ the morphism coming from this semi-direct product. If we consider $S(d,k)=A(d,k)\ltimes (N(d,k)\times N(d,k))$; given by $\overline \sigma \colon A(d,k) \to \text{Aut}(N(d,k)\times N(d,k))$, $\overline \sigma(a)=(\sigma(a),\sigma(a^{-1}))$, then we obtain that there exists a regular embedding from $G$ to $S(d,k)$, with $S(d,k)$ being an unimodular amenable group that has $\log t$ as isoperimetric profile by \cite{tesseraIsoperimetricprofilerandomwalksLCsolvablegroups2013}. We can now apply the theorem \ref{B} to obtain the desired corollary.
\end{proof}

Moreover, we have the following well-known proposition (see Sections 4.1.C and 4.1.D in \cite{GromovCoarseEmbeddings}, see also Lemma 2.2 in \cite{Frances}) that is an important source of coarse embeddings:
\begin{prop}
    Let's consider $G$ acting properly and by isometries on a compact pseudo-riemannian manifold $M$ of signature $(p,q)$, we then have that there exists a coarse embedding from $G$ to $O(p,q)$.
\end{prop}
\begin{proof}
    This is easy to see using the Proposition \ref{CoarseEmbedding}, that is we will provide a topological subgroup coupling $\Omega$ given by the frame bundle $\text{Fr}(M)$ of orthonormal basis under the pseudo-riemannian metric on $M$. It is clear that the $G$-action on $\text{Fr}(M)$ is proper and the $O(p,q)$-action on $\text{Fr}(M)$ is proper and cocompact, and that both actions commute with each other. This will provide us with a coarse embedding from $G$ to $O(p,q)$.
\end{proof}
As a consequence we have that:
\begin{cor}
    Given $G$ an amenable unimodular group that acts properly and by isometries on a compact pseudo-riemannian manifold, we then have that:
    \begin{equation*}
        \log t \preccurlyeq j_{p,G}(t)
    \end{equation*}
\end{cor}

\paragraph{\textbf{Outline of the Article}}
The second chapter is mainly devoted to introduce the necessary preliminaries, namely locally compact second countable group actions on Polish spaces, as well as defining quantitative Measure Subgroups and regular embeddings in the locally compact setting. The third chapter is devoted to the construction of the Measure Subgroup Coupling out of the existence of a regular embedding between locally compact groups. The final chapter is devoted to the introduction of the isoperimetric profile for locally compact groups as well as the proof of the monotonicity of the isoperimetric profile under certain quantitative measure subgroups.

\paragraph{\textbf{Acknowledgements}}
The author is grateful to his advisor Romain Tessera for his encouragement as well as suggestion of the problem.

\section{Preliminaries}

\subsection{Locally compact second countable group actions}

In this subsection we will introduce the necessary background to understand Borel actions, as well as fix notations needed in the article. In particular, we say that a Polish topological space is a separable completely metrizable topological space, and a standard Borel space is a measurable space $(X,\mathcal S)$ such that there exists a Polish topology on $X$ such that $\mathcal S$ is the $\sigma$-algebra generated by this topology.

Let's consider the following definitions:
\begin{defi}[Definition 2.1.8 in \cite{Zimmer}]
Given a measurable space $(X,\mathcal S)$, we say that this space is \textbf{countably separated} if there exists a countable collection of Borel sets $\{A_i\}_{i\in I}$ which separates points, that is for any $x,y$ in $X$, we have that $x$ is different from $y$ if and only if there exists $i\in I$ such that $x$ belongs to $U_i$ and $y$ doesn't belong to $A_i$.   
\end{defi}

\begin{defi}[Definition 2.1.9 in \cite{Zimmer}]
    Let's consider $G$ a locally compact second countable group and $X$ a standard Borel space. We say that a Borel action $G\curvearrowright X$ is \textbf{smooth} if $X/G$ with its natural measurable structure is countably separated.
\end{defi}

The main use that we will have from this definition is the following lemma:
\begin{lem}[Theorem A.7 in \cite{Zimmer}]
    Given $G$ a locally compact second countable group and $X$ a standard Borel space. If a Borel action $G\curvearrowright X$ is smooth, then $X/G$ is standard and there is a Borel section $s \colon X/G \to X$ of the natural projection. The image of $X/G$ under this section will be called a \textbf{ fundamental domain } of this action. 
\end{lem}

The following theorem due to Effros will be our main tool to prove that a given action is smooth:

\begin{lem}[Theorem 2.1.14 on \cite{Zimmer}]
Given $G$ a locally compact second countable group, $X$ a Polish topological space and $G\curvearrowright X$ a continuous action. The following assertions are equivalent:
\begin{enumerate}
    \item For every $x$ in $X$, the map $R_x\colon G/G_x\to G\cdot x$ is an homeomorphism.
    \item The Borel action $G\curvearrowright X$ is smooth.
    \item All orbits are locally closed; this is, they are the intersection of a open and a closed subset.
\end{enumerate}
\label{Effros}
\end{lem}

We have as well the following lemma that can be considered as a weaker "topological" version of the previous lemma:
\begin{lem}
    Given $\Om$ a locally compact second countable space, let's consider a proper cocompact continuous action $G\act \Om$; this is an action such that $p\colon G\times \Om \to \Om\times \Om$ given by $p(g,x)= (x,g\cdot x)$ is proper and such that there exists a compact subset $K$ of $\Om$ such that $G\cdot K= \Om$, we then have that this action is smooth and that $\Om/G$ is compact.
    \label{properaction}
\end{lem}
\begin{proof}
    Let's consider any $x$ in $\Om$, then since $p\colon G\times \Om \to \Om\times \Om$ is proper and $\Om$ is locally compact, then it is closed, which implies that $p(G\times\{x\})=\{x\}\times G\cdot x$ is closed, so the orbits are closed which by the lemma \ref{Effros} implies that the action is smooth.
    Moreover since the action is proper, we have that $\Om/G$ is a Haussdorff space and if we denote by $\pi \colon \Om \to \Om/G$ the continuous projection, we have that $\pi(K)= \Om/G$ which implies that $\Om/G$ is a compact Hausdorff space.
\end{proof}

Now given a smooth Borel action $G\act \Om$ and a fixed fundamental domain $\df_G$, there is a natural identification that will be denoted by $i_G\colon G\times \df_G \to \Om$ given by $i_G(g,x)=g\cdot x$. Note that $i_G$ is an isomorphism of measurable spaces.






We recall now the following useful lemma from \cite{Koivisto}:
\begin{lem}[Lemma 2.14 on \cite{Koivisto}]
    Given $G$ be a locally compact second countable group, $\la_G$ a fixed left-invariant Haar measure  and $\mathcal F$ a standard Borel space, let's consider the natural action of $G$ on $G\times \mathcal F$. We then have the following:
    \begin{enumerate}
        \item If $[\eta]$ is a probability measure class on $\df$ and $[\mu]$ is a $G$-invariant class of a $\sigma$-finite measure on $G\times \df$ such that it projects to $\eta$ under the projection $p_2\colon G \times \df \to \df$, we then have that $[\mu]= [\la_G\tp \ \eta]$
        \item If $\eta$ is a probability on $\df$ and $\mu$ is a $G$-invariant $\sigma$-finite measure on $G\times \df$ such that $[(p_2)_* \mu]=[\eta] $, this then implies that, there exists a measurable function $b\colon \df \to [0,\infty]$ such that $\mu = \la_G \tp b\eta$
        
    \end{enumerate}
    \label{action1}
\end{lem}
It is easy to see that this implies the following lemma:
\begin{lem}
 \label{action2} 
Given $G$ a locally compact second countable group, $\Om$ a Polish space, a free smooth action of $G$ on $\Om$ with fundamental domain $\df$ and $\mu$ a $G$-invariant $\sigma$-finite measure on $\Om$, then there exists a measure $\nu$ on $\df$ such that if we denote $i_G\colon G\times \df \to \Om$ the natural map coming from the action, we have that:
\begin{equation*}
  (i_G)_*(\la_G \tp \nu)= \mu. 
\end{equation*}
 \end{lem}
\begin{proof}
Consider $\eta$ a probability in the same measure class as $(p_2)_* (i_G)^*\mu$, and define $\nu$ as $b.\eta$ where $b$ is the function coming from the lemma \ref{action1}.
\end{proof}

To finish this subsection, we cite the following result from \cite{AdamsStuck}, needed to obtain free actions out of non-free actions:
\begin{lem}[Proposition 5.3 on \cite{AdamsStuck}]
    Given a locally compact second countable group $G$, there exists a compact metrizable topological space $X$ and $G\curvearrowright X$ a continuous free action. 
    \label{free}
\end{lem}

\subsection{Quantitative Measure Couplings of locally compact groups}


In this subsection we will extend the notion of quantitative measure subgroup coupling given in \cite{Romain} to the locally compact setting. Moreover, we will always consider $H,G$ two locally compact compactly generated groups with $S_H,S_G$ two symmetric compact generating subsets of $H,G$ respectively. Let's denote by $|\cdot|_{H},|\cdot|_{G}$ the word length on $H,G$ with respect to $S_H,S_G$ respectively. This induces compatible metrics $d_H,d_G$ on $H,G$ respectively that are left-invariant, proper and quasi-isometric to the word length in each respective group.
\begin{Convention}
    In this subsection, following \cite{Romain}, we will denote by $*$ the smooth actions and by $\cdot$ any other actions.
\end{Convention}
\begin{defi}
 We say that $(\Om,\mu)$ is a \textbf{measure subgroup coupling} from $H$ to $G$ if we have two smooth commuting measure preserving actions from $H,G$ on $(\Om,\mu)$ with fundamental domains $\df_H,\df_G$ respectively, such that:
\begin{enumerate}
    \item The $G$-action is free;
    \item The $H$-action is free; 
    \item There is a $\sigma$-finite measure $\nu_H$ on $\df_H$ and a finite measure $\nu_G$ on $\df_G$ such that, the following maps are isomorphisms of measured spaces 
    \begin{align*}
        i_G&\colon (G\times \df_G,\mu_G \otimes \nu_G) \to (\Om,\mu)  \\
        i_H&\colon (H\times \df_H,\mu_H \otimes \nu_G) \to (\Om,\mu).
    \end{align*}
    
\end{enumerate}
If there exists a measure subgroup coupling from $H$ to $G$ we say that $H$ is a measure subgroup of $G$. 

\label{coupling}

Let's define as well the \textbf{induced action} of $G$ on $\df_H$ as the natural action of $G$ over $\Om/H$ viewed under the identification $s_H\colon \Om/H\to\df_H$. In a similar manner, we define as well the induced action of $H$ on $\df_G$, and given any $g$ in $G$ and any $x$ in $\df_H$ we will denote by $g\cdot x$ this induced action. On the other hand the $G,H$-actions on $\Om$ will be denoted by $g*x,h*x$ for $g$ in $G$,$h$ in $H$ and $x$ in $\Om$.
    
Let's define the \textbf{cocycle} $c_G\colon H\times \df_G \to G$ by the following:
\begin{equation*}
    h*x=c_G(h,x)*(h\cdot x)
\end{equation*}

\end{defi} 

Now, the definition of quantitative measure subgroups(cf. \cite[section 2.4, page 16]{Romain}) becomes natural to extend, which we define as:
\begin{defi}(Integrability of cocycles)
For any $p$ in $[1,\infty]$, we say that $(\Om,\mu)$ is an \textbf{$L^p$-measure subgroup coupling} from $H$ to $G$ if $(\Om,\mu)$ is a measure subgroup coupling from $H$ to $G$ and also we have that:
\begin{equation*}
    \sup_{h\in S_H} \lVert|c_G(h,\cdot)|_G\rVert_p  < \infty
\end{equation*}
where $|c_G(h,\cdot)|_G$ belongs to $L^0(\df_G,\nu_G)$, since $\nu_G$ is a finite measure, we have that $L^p$ implies $L^q$ for every $p>q$. Note that this is the same as saying that for any $h$ in $H$, we have that $|c_G(h,\cdot)|_G$ belongs to $L^p(\df_G,\nu_G)$(cf.\cite[Appendix A.2]{BaderFurmanSauer})

Moreover, for a function $\varphi\colon (0,\infty)\to (0,\infty)$, such that $\varphi$ and $t\to \frac{t}{\varphi(t)}$ are non-decreasing, we say that $(\Om,\mu)$ is a $\varphi$-measure subgroup coupling from $H$ to $G$ if $(\Om,\mu)$ is a measure subgroup coupling from $H$ to $G$ such that:
\begin{equation*}
    \sup_{h\in S_H} \varphi(|c_G(h,\cdot)|_G) < \infty
\end{equation*}
Note that this definition doesn't depend on the choice of the compact generating set, since $t\mapsto \frac{t}{\varphi(t)}$ is non-decreasing, we have that $\varphi(ct)\leq c\varphi(t)$ for $c>1$ and $t>0$.
Moreover if $\varphi$ is subadditive, this is the same as saying that $\varphi(||c_G(h,\cdot)||)$ is finite for all $h$ in $H$ without any uniform condition on the bound by the same argument in \cite[Appendix A.2]{BaderFurmanSauer}. In particular, this is the case for $\varphi(t) = t^{p}$ with $0<p<1$.
%

\end{defi}

\subsection{Regular Embeddings and Discretizations}
In order to pass from the discrete setting to the non-discrete one, it is usually helpful to define the following notions:
\begin{defi}
    Given $(X,d)$ a metric space and a parameter $s>0$, we say that $Y\subset X$ is $s$-\textbf{discrete} if $d(x,y)\geq s$ for all $x,y$ in $Y$ such that $x\neq y$. We say that $Y\subset X$ is $s$-\textbf{dense} if for all $x$ in $X$, there exists $y\in Y$ such that $y \in B(x,s)$. We say that $Y\subset X$ is a $s$-\textbf{discretization} if it is a maximal $s$-discrete subset.
\end{defi}

By the maximality condition in the definition of $s$-discretization, we can see that an $s$-discretization is always $s$-dense as well. More precisely we have that:
\begin{lem}
    \label{skeleton}
    Given any set $A$ in a metric space $(X,d)$, consider $C$ a maximal $r$-discrete subset of $A$. We then have that:
    \begin{equation*}
        A\subset [C]_r=\bigcup_{c\in C} B(c,r)
    \end{equation*}
\end{lem}
\begin{proof}
    Let us suppose that there is an $a$ in $A$ not belonging to $[C]_r$, this implies that $d(a,c)\geq r$ for all $c$ in $C$, we can then consider $C\cup a$ as a $r$-discret subset of $A$ that contains $A$ contradicting the maximality of $A$.
\end{proof}

\begin{defi}
Given a metric space $(X,d)$ we say it is \textbf{balanced} if for any two $r_1>r_2>0$, there exists $N$; depending only on $r_1,r_2$, such that we can cover any $r_1$-ball by at most $N$ $r_2$-balls.
\end{defi}

Let's consider the following generalization of regular embeddings of graphs to non-discrete spaces given in \cite{Poincare}:
\begin{defi}[Subsection 1.1, p.4 in \cite{Poincare}]
Consider $X,Y$ two balanced metric spaces, a map $f\colon X \to Y$ is called a \textbf{regular embedding}  if there exists scales $r,s>0$, $L>0$ and $N\in \mathbb N$  such that:
\begin{itemize}
    
    \item $d(f(x),f(x'))\leq Ld(x,x')+L$ for all $x,x'$ in $X$ and
    \item the preimage of each $r$-ball in $Y$ can be covered by at most $N$ $s$-balls on $X$
    
\end{itemize}
It is clear from the definition that quasi-isometries are regular embeddings. Moreover since we are working with balanced metric spaces, there is no importance on the scale we use, that is, if a map $f$ is regular at scales $(r,s)$ it is regular at all scales.
It is also clear that the composition of regular embeddings is also regular. 
\end{defi}

Note as well that for $G$,$H$ locally compact compactly generated groups, $(G,d_G)$ and $(H,d_H)$ are balanced metric spaces; for that reason we can fix the scales used in each group, we will consider $s$-discretizations on $H$ and $3$-discretizations on $G$.
Now that we have fixed the scales on $G,H$ we can define the analog of Definition 4.1 in \cite{Romain}:
\begin{defi}
    Given $(\Om,\mu)$ a measure subgroup coupling from $H$ to $G$, we say that $(\Om,\mu)$ is  \textbf{coarsely $m$ to $1$} if for any $x$ in $\df_G$, the map $c_x\colon  H \to G$ given by $c_x(h)= c_G(h^{-1},x)$ satisfies that the preimage under this map of any ball of radius $3$ in $G$  can be covered but at most $m$ balls of radius $s$ in $H$.
\end{defi}
We will be interested on the following definition as well:
\begin{defi}
    We say that $A\subset (X,d)$ is \textbf{$r$-thick} if it is the union of balls of radius at least $r$. Since each ball of radius at least $r$ can be covered by balls of radius $r$, this is the same as saying that there exists $C$ such that $A=[C]_r$.
\end{defi}

\section{Construction of Measure Subgroup Couplings}

Let's first prove the following proposition:
\begin{prop}
    Let $f\colon H \to G$ be a regular embedding between unimodular locally compact compactly generated groups, if $H$ is amenable, we then have that there exists $L^\infty$-measure subgroup coupling $(\Om,\mu)$ from $H$ to $G$ that is coarsely $m$ to $1$ for some $m$.
\end{prop}

\begin{proof}


Let's fix $Y$  a $s$-discretization of $G$ and $Z$ a $3$-discretization of $H$, let's denote by $\pi\colon G \to Z$ a retraction of $G$ to $Z$, that takes every element $g$ in $G$ to some  $\pi(g)$ in $Z$ with $d(g,\pi(g))\leq 3$. Up to considering $\pi \circ f$ instead of $f$ we can suppose that $f(Y)\subset Z$.

Since $f$ is regular, it's clear as well that for every $z$ in $Z$, the number of elements of $\{y\in Y: f(y)=z\}$ is uniformly bounded by $T$, for some $T>0$.

Now as in \cite{Rosendal} we will obtain a $C$-Lipschitz map $\eta \colon H \to \Delta(G) \subset L^1(G,\la_G)$ where $\Delta(G)$ will be a large-scale analog of $G\times \Delta$ where $\Delta$ is a finite dimensional simplex and we will build our coupling space $(\Omega,\mu)$ with the help of this map. 

More specifically, we can find $(\beta_y\colon H \to [0,1])_{y\in Y}$ a family of $M$-Lipschitz functions for some $M>0$ such that $\supp \beta_y = \overline B(y,s+1)$ for all $y$ in $Y$  and such that for all $h$ in $H$:
\begin{equation*}
    \sum_y \beta_y(h) = 1
\end{equation*}

This is done in \cite{Rosendal}, by considering for all $y$ in $Y$, $\theta_y\colon H \to [0,s+1]$ given by $\theta_y(h) = \max\{0,s+1-d_H(y,h)\}$, it is clear that $\theta_y$ is $1$-Lipschitz and $\theta_y\geq 1$ on $B(y,s)$ and $\theta_y=0 $ outside of $B(y,s+1)$. It follows that $\Theta(h) = \sum_{y \in Y} \theta_y(h)$ is a bounded Lipschitz function with $\Theta\geq 1$. We can get then $\beta_y$ by defining as $\beta_y = \theta_y/\Theta$ for all $y$ in $Y$.
 
Let's define then the family of functions $(\alpha_z)_{z\in Z}\colon H\to [0,1]$ by:
\begin{equation*}
    \alpha_z(h) = \sum_{f(y)=z} \beta_y(h)
\end{equation*}

Since for every $z$ in $Z$, the number of elements of $\{y\in Y: f(y)=z\}$ is uniformly bounded by $T$, we have that $\alpha_z$ is $TM$-Lipschitz for all $z$ in $Z$. For all $h$ in $H$, we denote by $Y_h=\{y\in Y : \beta_y(h) > 0 \}$ and $Z_h =\{z\in Z : \alpha_y(h) > 0 \}$, we clearly have that $Z_h\subset f(Y_h)$ which implies that the diameter of $Z_h$ is bounded by $L(\diam(B(y,s+1)))+L=L(2s+2)+L=R$. Now consider $N$ to be the maximum number of a $3$-discrete subset of $G$ of diameter bounded by $R$.

Let's define then $\eta\colon H \to L^1(G,\la_G) $ given by:
\begin{equation*}
    \eta(h)= \sum_{z\in Z} \alpha_z(h) \chi_{zB} 
\end{equation*}
where $B$ is the unitary closed ball in $G$, and define $\Delta(G)\subset L^1(G,\la_G)$ as:
\begin{equation*}
     \Delta(G)=\left\{\sum_{i=1}^{m} \alpha_i \chi_{z_iB} : 
     \{z_1,\ldots,z_m\} \text{ is } \text{3-discrete}
                                         \text{ of diameter} \leq R,
                                         \sum \alpha_i=1, \alpha_i\geq 0
                                         \right\}
\end{equation*}
Since $Z_h$ has diameter bounded by $R$ it is clear then that $\eta(h)$ belongs to $\Delta(G)$ for all $h$ in $H$.
Furthermore; in Claim 3,p.3 in \cite{Rosendal}, the authors proved the following claim:
\begin{claim}
$\Delta(G)$ is locally compact in the $L^1$-topology in $L^1(G,\mu_G)$. In fact:
\begin{equation*}
    [K,\eps]=\{\xi \in \Delta(G): \langle \xi | \chi_K\rangle\geq \eps\}
\end{equation*}
is compact for every compact set $K\subset G$ and $\epsilon > 0$. Moreover, every compact set is contained in some $[K,\eps]$.
\end{claim}

Let's consider the space of maps $\Delta(G)^H$ equipped with the product topology, endowed with two natural commuting actions of $G$ and $H$ given by:
\begin{align*}
(g*\xi)(h)&=\la(g)\xi(h)\\
(h_1* \xi)(h)&=\xi(h_1^{-1}h).
\end{align*}
where $g$ belongs to $G$, $h,h_1$ belong to $H$, $\xi$ belongs to $\Delta(G)^H$ and $\la$ is the left regular representation of $G$ on $L^1(G,\la_G)$. We then set $\Omega_0$ as the subset of $\Delta(G)^H$ given by $\overline{(G\times H)*\zeta }$.

Note that $\zeta$ is a Lipschitz function, to be more precise:
\begin{align*}
    \left\Vert\zeta(h)-\zeta(h')\right\Vert_1&= \left\Vert\sum_{z\in Z_h\cup Z_{h'}} (\alpha_z(h)-\alpha(h')) \cdot\chi_{zB}\right\Vert_1\\
    &\leq  \sum_{z\in Z_h\cup Z_{h'}} |\alpha_z(h)-\alpha(h')| \cdot \lVert\chi_{zB}\rVert_1\\
    &\leq 2NTMd_H(h,h').
\end{align*}
Since $d_H$ is left-invariant, we have that any $\xi$ in $(G\times H)*\zeta$ is also $2NTM$-Lipschitz. This then implies that any $\xi$ in $\Om_0$ is also $2NTM$-Lipschitz.

Note as well, that for any $h,h'$ in $H$, we have the following:
\begin{align*}
    d_G(\supp(\zeta(h)),\supp(\zeta(h')))&\leq 2+ d_G(Z_h,Z_{h'})\\
    &\leq 2+L(d_H(Y_h,Y_{h'}))+L\\
    &\leq 2+L(d_H(B(h,s+1),B(h',s+1)))+L\\
    &\leq 2+L(2s+2+d_H(h,h'))+L.
\end{align*}
It can be seen as well that the same is true for any $\eta$ in $\Om_0$.
 
We will consider then $\Om=X\times \Om_0$, where $X$ is the compact topological space equipped with a free continuous $G\times H$-action given by the Lemma \ref{free}. This makes that there exists two natural commuting $G,H$-actions on $\Om$ that are now free. By an abuse of notation, we will still by $*$ these two actions on $\Om$.
We claim then that:
\begin{claim}
    $\Om$ is a locally compact space.
\end{claim}
\begin{proof}
    It is clear that we need only to prove that $\Om_0$ is a locally compact topological space. This follows from the same arguments as \cite{Rosendal}, to be more precise, given any $\xi$ in $\Om_0$, if we denote by $K$ the support of $\xi(1_H)$ we then have that:
 \begin{equation*}
        K_{\xi}= \left\{\eta \in \Om_0 : \eta(1_H)\in [K,1/2]\right\}
    \end{equation*}
   is a compact neighbourhood of $\xi$, since for any $\eta$ in $K_\xi$ we have that  $\supp(\eta(1_H))\cap K\neq \emptyset$, which then implies that for any $h$ in $H$ we have that $\supp(\eta(h))\subset [K]_{3(R+2)+L(d_H(1,h))} $ and that $\eta(h)\in [[K]_{3(R+2)+L(d_H(1,h))},1]$. 
\end{proof}

\begin{claim}
$\Om$ is a second countable topological space.
\end{claim}
\begin{proof}
    This follows from the same arguments as \cite{Koivisto}, since $\Delta(G)\subset L^1(G,\la_G)$ is a separable space and $\Om_0\subset C(H,\Delta(G))$ is locally compact and cosmic(see lemma 6.18 and its proof on \cite{Koivisto} for the definition), we have that $\Om_0$ is a locally compact second countable Hausdorff topological space. This implies then that $\Om$ is second countable.   
\end{proof}
 
\begin{claim}
The action $H\act \Om$ is continuous.
\end{claim}
\begin{proof}
    It is only needed to prove that $H\act \Om_0$ is continuous, this follows from the same arguments as in \cite{Rosendal}.
\end{proof}
\begin{claim}
 The action $G\act \Om$ is continuous, proper and cocompact.   
\end{claim}
\begin{proof}
    In order to prove that $G\act \Om$ is continuous, we need only to prove that $G\act \Om_0$ is continuous; this follows from the same arguments as in \cite{Rosendal}.
    
    In order to prove that $G\act \Om$ is proper, we need to prove that $G\act \Om_0$ is proper, let's consider $\tilde K$ a compact subset of $\Om_0\subset \Delta(G)^H$, since the projection onto the $1_H$-coordinate is continuous, we can suppose there exists $K$ a compact subset of $G$ and $\epsilon>0$ such that:
    \begin{equation*}
        \tilde K\subset \{\xi\in \Om_0: \xi(1_H) \in [K,\epsilon]\}
    \end{equation*}
    Now, for any $g$ in $\{g \in G : g\tilde K \cap \tilde K \neq \emptyset\}$ we have that there exists $\xi$ in $\Om_0$ such that $\xi(1_H)\in [K,\epsilon]\cap [gK,\epsilon]$. We thus have that $\supp\ \xi(1_H)\cap K\neq \emptyset$ and $\supp\ \xi(1_H)\cap gK \neq \emptyset$, which allows us to prove that $d_G(1,g)\leq 2\diam(K)+\diam(\supp(\xi(1_H)))\leq 2\diam(K) + R+2$. So we conclude that $G\act \Om$ is proper.
    
    In order to prove that $G\act \Om$ is cocompact, let's consider $B(1_H,2R)$, given any $\xi$ in $\Om_0$ since $\supp\ \xi(1_H)$ has diameter bounded by $R+2$, there exists $g$ in $G$ such that $\supp (g*\xi)(1_H)\subset B(1_H,R+2)$, so if we define:
    \begin{equation*}
        C= \{\xi \in \Om_0: \xi(1_H)\in [B(1_G,R+2),1]\}
    \end{equation*}
    we have that $G*C=\Om_0$ and that $C$ is a compact neighbourhood by Claim $2$. This then implies that $G*(X\times C)= \Om$  

\end{proof}

\begin{claim}
The action $H\act \Om$ is a smooth action.
\label{Important}
\end{claim}
\begin{proof}
    In order to prove this claim, we will need to prove that all $\xi$ in $\Om_0$ will be proper maps, more specifically we have that:
\begin{prop}
For all $\xi$ in $\Om_0$ we have that for any $K$ of diameter bounded by $r$, there exists $N(r)$  only depending on $r$ such that $\xi^{-1}([K,\epsilon])$ can be covered by at most $N(r)$ balls of radius $s$. 
\label{regularitylemma}
\end{prop}
\begin{proof}
    Let's suppose that $\xi_n $ in $(G\times H)* \zeta $ converges to $\xi$ in $\Om_0$, it is possible to see that for each $r>0$ there is a $N(r)$ such that all $\xi_n^{-1}([K,\epsilon/2])$ can be covered by at most $N(r)$ balls of radius $s/2$. 
    Let's consider then $C\subset H$ a $s$-discrete maximal subset of $\xi^{-1}([K,\epsilon])$, suppose that $C$ has cardinality at least $N(r)+1$ and let's denote $F\subset C$ some finite subset of cardinality $N(r)+1$. We then have that since $\langle\chi_K | \xi\rangle \geq \epsilon $ and $\xi_n(c)\to \xi(c)$ for all $c\in F$, there exists $n_c$ such that for all $n\geq n_c$ we have that $\langle\chi_K | \xi_n(c)\rangle \geq \epsilon/2$. If we consider $m = \max_F {n_c}$ we get that $c\in \xi_m^{-1}([K,\epsilon/2])$ for all $c\in F$. Since $F$ is s-discrete we have that $|F|\leq N(r)$ which is a contradiction. Therefore, $C$ is finite and moreover $|C|\leq N(r)$ which implies that $\xi^{-1}([K,\epsilon])$ can be covered by at most $N(r)$ balls of radius $s$ by the Lemma \ref{skeleton}.
    \label{regularity}
\end{proof}
Moreover, since $[K,\eps]$ is a basis of $\Delta(G)$, Proposition \ref{regularity} implies as well that any $\xi$ in $\Om_0$ is proper when we see it as a map from $H$ to $\Delta(G)$. 
Furthermore, for any $(x,\xi)$ in $\Om$ we have that the map:
\begin{equation*}
    R_{(x,\xi)} \colon H \to H*(x,\xi)
\end{equation*}
given by $R_{(x,\xi)}(h)=(h\cdot x,h* \xi)$; is a proper map; this is so because any $\tilde K$ compact set in $\Om$ is included on $X\times V$, with $V=\{\xi \in \Om : \xi(1_H)\in [K,\eps] \}$ for some $K$ compact subset of $G$ and $\epsilon>0$. Then we have that $R_{(x,\xi)}^{-1}(X\times V)=\{h\in H: \xi(h^{-1})\in [K,\eps]\}= \left(\xi^{-1}([K,\eps])\right)^{-1}$ which is compact. This implies that $R_{(x,\xi)}$ is an homeomorphism (cf. the corollary in \cite{Palais}). Since by the claims $3,4$ we have that $\Om$ is a Polish topological space, we can then use the Lemma \ref{Effros} to deduce that $H \act \Om$ is a smooth action. 
\end{proof}



Now by the lemma \ref{properaction}, we have that the $G$ action is smooth, so let's fix then $\df_G$ and $\df_H$ fundamental domains of the respective $G$-action and $H$-action respectively. Now in order to construct the measure $\mu$ on $\Om$, let's consider the continuous induced action $H\act \Om/G$, since the $G$-action is proper cocompact, by lemma \ref{properaction} we have that $\Om/G$ is compact, which since $H$ is amenable, implies that there exists $\nu$ probability measure on $\Om/G$ that is $H$-invariant, we then define $\nu_G$ as the probability on $\df_G$ given by the identification $s_G\colon \Om/G\to \df_G$, and then define the measure $\mu$ as $(i_G)_* (\la_G\tp \nu_G)$. We then have that:
\begin{claim}
    $\mu$ is $H$-invariant
\end{claim}
\begin{proof}
    Given any $h$ in $H$, and any measurable  $B\subset G, F\subset \df_G $, we need to prove:
    \begin{equation*}
        \mu(h^{-1}*{i_G(B\times F)})= \mu(i_G(B\times F))
    \end{equation*}
    Which is equivalent to:
    \begin{align*}
        \mu(h^{-1}*{i_G(B\times F)})&=\int_{\Om} \chi_{h^{-1}*{i_G(B\times F)}} d\mu      \\
        &=\int_{G\times \df_G}\chi_{B\times F}(i_G^{-1}(h*g*x)) d\la_G(g),d\nu_G(x) \\
        &=\int_{G\times \df_G}\chi_{B\times F}(g\cdot c_G(h,x),h\cdot x) d\la_G(g) d\nu_G(x)\\
        &=\int_{\df_G} \la_G(B)\chi_{F}(h\cdot x)  d\nu_G(x)\\
        &=\la_G(B) \nu_G(F).
    \end{align*}
    Which implies that $\mu$ is $H$-invariant.
    \end{proof}
Now since $\mu$ is $H$-invariant, by Lemma \ref{action2}, there exists $\nu_H$ a measure on $\df_H$ such that:
\begin{equation*}
    (i_H)_*(\la_H\tp \nu_H)= \mu
\end{equation*}
This then implies that $(\Om,\mu)$ is a measure subgroup coupling from $H$ to $G$. Moreover, we have that:
\begin{claim}
    $(\Om,\mu)$ is a $L^\infty$-measure subgroup coupling.
\end{claim}
\begin{proof}
Since the $G$-action is cocompact, we can assume that $\df_G$ is precompact. Now, since $S_H$ is precompact, we have that $S_H*\df_G$ is precompact as well since the $H$-action is continuous. Furthermore, as the $G$-action is proper we have that:
\begin{equation*}
    \{g \in G : g*\df_G \cap S_H*\df_G \neq \emptyset \}
\end{equation*}
is precompact, let's say that it is included in $B_G(1_G,L)$ for some $L>0$. Now for any $(x,\xi)$ in $\df_G$ and any $h$ in $S_H$, if we denote $g=c_G(h,x,\xi)$, we have that $h*(x,\xi)=g*(h\cdot (x,\xi))$ which implies that $g$ belongs to $B_G(1_G,L)$, since $d_G$ and $|\cdot|_G$ are quasi-isometric, this implies that there exists $C>0$ such that $|c_G(h,x,\xi)|_G\leq C$ for all $(x,\xi)$ in $\df_G$, which implies that $(\Om,\mu)$ is an $L^\infty$-measure subgroup coupling.

        
 
\end{proof}

\begin{claim}
    $(\Om,\mu)$ is a coarsely $m$-to-$1$ measure subgroup coupling for some $m>0$.
\end{claim}
\begin{proof}
Since $\df_G$ is precompact, the projection of $\df_G$ onto the coordinate $1_H$ is precompact as well, so there exists $[B(1,d),\epsilon]$ such that $\df_G\subset X\times \{\xi\in \Om_0 : \xi(1_H) \in [B(1,d),\epsilon ]\}$. Now given any $(x,\xi)$ in $\df_G$ and any $g$ in $G$, let's consider $S$ to be  $c_G(\cdot,x,\xi)^{-1}(B(g,3))$. Now for any $h$ in $S$, there exists some $\overline g$ in $B_G(g,3)$ such that:
\begin{equation*}
h*(x,\xi)=\overline g*(h\cdot (x,\xi))   
\end{equation*}
This then implies that $\xi(h^{-1})\in \la(\overline g)([B_G(1_G,d),\epsilon])\subset [B_G(g,d+3),\epsilon]$ which by proposition \ref{regularitylemma} implies that there exists $m$ such that $S^{-1}$ can be covered by at most $m$ balls of radius $s$. This then implies that $(\Om,\mu)$ is coarsely $m$ to $1$.

\end{proof}
All these claims prove our desired proposition.
\end{proof}
We can now prove the announced theorem at the beginning of the article:
\begin{teo}
\label{teoprincipal}
Given $H$,$G$ two locally compact compactly generated unimodular groups, with $H$ being amenable, the following are equivalent:
\begin{enumerate}
    \item There exists a regular embedding from $H$ to $G$.
    \item There exists a $L^\infty$-measure subgroup coupling $(\Om,\mu)$ from $H$ to $G$ that is coarsely $m$ to $1$. 
\end{enumerate}
\end{teo}
\begin{proof}
    The fact that $(1)$ implies $(2)$ comes from the last proposition, so we only need to prove the converse. Suppose then that we have $(\Om,\mu)$ a $L^\infty$-measure subgroup coupling from $H$ to $G$ that is coarsely $m$ to $1$. Since this implies that for any $x$ in $\df_G$ the map $c_x\colon H \to G$ given by $c_x(h)=c_G(h^{-1},x)$ satisfies the second condition on the definition of a regular embedding, all we need to prove is that this map satisfies the first condition, now since the coupling is $L^\infty$ we know there exists some $C>0$ such that for any $h$ in $S_H$ and any $x$ in $\df_G$:
    \begin{equation*}
        ||c(h,x)||_G\leq C
    \end{equation*}
    Using this and the cocycle condition, for any $h$ in $H$, we then have that:
    \begin{equation*}
        ||c_G(h,x)||_G\leq C||h||_H
    \end{equation*}
    Then for any $h_1,h_2$ in $H$, we then have that $||c_x(h_1)||_G\leq C||h_1||_H$ and that     $||c_x(h_2)||_G\leq C||h_2||_H$, which then implies that:
    \begin{equation*}
        ||c_x(h_1)^{-1}c_x(h_2)||_G\leq C||h_1||_H+C||h_2||_H
    \end{equation*}
     Since $d_G$ is quasi-isometric to $|\cdot|_G$ and $d_H$ is quasi-isometric to $|\cdot|_H$ we then have that there exists $L$ such that:
     \begin{equation*}
      d_G(c_x(h_1),c_x(h_2))\leq Ld_H(h_1,h_2)+L   
     \end{equation*}
     Which proves the desired theorem.
\end{proof}

\section{Monotonocity of the Isoperimetric Profile}

\subsection{Isoperimetric profile of locally compact groups}
The goal of this subsection is to recollect the relevant properties of the isoperimetric profile for locally compact groups  as well as some relevant lemmas that will be useful in the proof of the monotonicity of the isoperimetric profile under quantitative measure couplings, namely we recall the fact that in order to understand the large-scale behaviour of the isoperimetric profile we just need to restrict ourselves to thick subsets. Now, given $G$ a locally compact unimodular group with compact generating symmetric subset $S_G$, with fixed Haar measure $\la_G$; acting by isometries on a $L^p$-space $E$, we define the $p$-gradient of a function $f$ in $E$ as:
        \begin{equation*}
            ||\nabla_G f||_p = \sup_{s\in S_G} ||f-s*f||_p
        \end{equation*}
The following classical lemma will be useful in the next subsection, so we prove it here:

\begin{lem}Given $E$ an $L^p$ space, and $f$ in $E$, for any  $g$ in $G$, we have the following:
    \begin{equation*}
        ||f-g\ast f||_p \leq |g|_G ||\nabla_G f||_p
    \end{equation*}
\label{ineq}
\end{lem}
\begin{proof}
 By the definition of $|g|_G=n$, we have that $g=s_1\cdot s_2\cdots  s_n$ where each $s_i$ belongs to $S_G$ for $i$ in $\{1,\cdots, n\}$. It is clear then that:
 \begin{align*}
     ||f-g\ast f||_p &= \sum_{i=0}^{n-1} ||s_1\cdots s_i \ast f- s_1\cdots s_{i+1}\ast f||_p \\
     &= \sum_{i=0}^{n-1} || f- s_{i+1}\ast f||_p\\
     &\leq n ||\nabla_G f||_p= |g|_G ||\nabla_G f||_p
 \end{align*} 
 Concluding the proof of the lemma.
 \end{proof}

    We will need as well the following sublinear version of the previous lemma:
 \begin{lem}
 Given $E$ a $L^1$-space, $f$ in $E$ and $\varphi\colon (0,\infty)\to (0,\infty)$ a function such that $\varphi$ and $t\mapsto \frac{t}{\varphi(t)}$ are non-decreasing functions. Then for any $g$ in $G$ we have the following:
\begin{equation*}
    ||f-g*f||_1\leq 2\varphi(|g|_G ||\nabla_G f||_1) \dfrac{||f||_1}{\varphi(||f||_1)} 
\end{equation*}
  \label{ineq2}
 \end{lem}
\begin{proof}
   By the Lemma \ref{ineq}, we obtain that $||f-g*f||_1\leq |g|_G||\nabla_G f||_1$. We have as well that $||f-g*f||_1\leq 2||f||_1$. We then obtain that:
   \begin{align*}
       ||f-g*f||_1 &\leq \varphi(||f-g*f||_1) \frac{||f-g*f||_1}{\varphi(||f-g*f||_1)} \\
       &\leq \varphi(|g|_G.||\nabla_G f||_1) \frac{2||f||_1}{\varphi(2||f||_1)}\\
       &\leq  2\varphi(|g|_G.||\nabla_G f||_1) \frac{||f||_1}{\varphi(2||f||_1)}
       \leq 2\varphi(|g|_G.||\nabla_G f||_1) \frac{||f||_1}{\varphi(||f||_1)}
   \end{align*}
   Concluding the proof of the lemma.
 \end{proof}
 Now, in order to define the $L^p$-isoperimetric profile of $G$, we will consider either the gradient coming from the left regular representation $\la:G\act L^p(G)$, which will be denoted by $||\nabla_G^l f||_p$ or the right regular representation $\rho\colon G\act L^p(G)$ which will be denoted $||\nabla_G^r f ||_p$. 

Now, given any subset $A$ of $G$ of finite measure, we define the right $L^p$-isoperimetric profile of $A$ as:
\begin{equation*}
    J_{p,G}^r(A) = \sup_{f\in L^p(A)} \dfrac{||f||_p}{||\nabla^r_G f ||_p}     
\end{equation*}
and we define the right $L^p$-isoperimetric profile of $G$ as:
\begin{equation*}
    j_{p,G}^r(v)= \sup_{\la_G(A)\leq v} J_{p,G}^r(A)
\end{equation*}
We define in a similar manner the left isoperimetric profile. Note that, since $G$ is unimodular, the right and left isoperimetric profiles coincide. Moreover, it can be seen as well, that in order to compute $J^r_{p,G}(A)$, we can restrict ourselves to functions in $L^\infty(A)$. More precisely:
\begin{prop}
If we define:
\begin{equation*}
    \overline J_{p,G}(A)= \sup_{f\in L^\infty(A)} \dfrac{||f||_p}{||\nabla_G^r f||_p} 
\end{equation*}
we obtain that $\overline J_{p,G}(A)=J_{p,G}(A)$. 
\end{prop}
\begin{proof}
    It is clear that $J_{p,G}(A)\geq \overline J_{p,G}(A)$. In order to prove the inverse inequality, let's consider any $f$ in $L^p(A)$, there exists a sequence $(f_n)_{n\in \mathbb N}$ in $L^\infty(A)$ such that $f_n\to f$ in $L^p(A)$. For any $s$ in $S_G$, we have that:
    \begin{equation*}
        \big| ||f_n-\rho(s)f_n||_p-||f-\rho(s)f||_p \big|\leq 2||f_n-f||_p
    \end{equation*}
    So, for any $\epsilon>0$ there exists $n_0$, such that for all $n\geq n_0$ and all $s\in S_G$, we have that:
    \begin{equation*}
        ||f-\rho(s)f||_p-\epsilon \leq ||f_n-\rho(s)f_n||_p\leq ||f-\rho(s)f||_p+\epsilon
    \end{equation*}
    Which then implies that for all $n\geq n_0$ we have that:
    \begin{equation*}
        ||\nabla_G f||_p -\epsilon \leq ||\nabla_G f_n||_p \leq ||\nabla_G f||_p+\epsilon
    \end{equation*}
    So, $||\nabla_G f_n||_p\to ||\nabla_G f||_p$ when $n\to \infty$. Then, it's clear to see that:
    \begin{equation*}
        ||f_n||_p\leq \overline{J}_{p,G}(A) ||\nabla_G f_n||_p.
    \end{equation*}
    Which then implies that:
    \begin{equation*}
        ||f||_p\leq  \overline{J}_{p,G}(A) ||\nabla_G f||_p.
    \end{equation*}
    For all $f$ in $L^p(A)$, so $J_{p,G}(A) \leq \overline{J}_{p,G}(A) $.
\end{proof}

Furthermore, we will need to restrict ourselves to thick subsets as well, in order to do that, we will use the other two natural notions of gradients introduced in \cite{IsoProfilLC}. More precisely, for any $h>0$, any subset $A$ of finite measure and any $f\in L^\infty(A)$ we obtain the following 3 notions of gradients:
\begin{align*}
    ||\nabla^1_h f||^p_p&= \int_{G} \sup_{s\in B(1,h)}\left|f(gs)-f(g)\right|^p d\la_G(g)\\
    ||\nabla_h f ||_p^p &= \sup_{s\in B(1,h)} \int_{G}  |f(gs)-f(g)|^p d\la_G(g) \\
    ||\nabla^2_h f||_p^p&= \int_G \fint_{s\in B(1,h)} \left| f(gs)-f(g)\right|^p d\la_G(s) d\la_G(g) .
\end{align*}

The gradient $||\nabla^1||_p$ was introduced in \cite{IsoProfilLC}, p.7, as $|\nabla f|_h$; while the gradient $||\nabla^2||_p$ was introduced in the same article as $|\nabla_{P,p}|$ for $dP_x= \frac1{\la_G(B(1,h))} 1_{B(x,h)} d\la_G$. Now, it is easy to see that we have the following inequalities:
\begin{equation*}
    ||\nabla^1_h f||_p \geq ||\nabla_h f||_p \geq ||\nabla_h^2 f||_p
\end{equation*}

Now, it is proven in the Proposition 7.2 in \cite{IsoProfilLC} that the asymptotic behaviour of the respective isoperimetric profiles under $\nabla^1$ and $\nabla^2$ coincide, and in the Proposition 10.1 in \cite{IsoProfilLC} that the isoperimetric profile under $\nabla^1$ doesn't depend on the parameter $h$, provided $h$ is large enough.
Moreover, the following lemma is proven in the same article:
\begin{lem}[Proposition 8.3 in \cite{IsoProfilLC}]
    There exists $C>0$ such that for any $f$ in $L^\infty(A)$, there is a function $\tilde f$ in $L^\infty(\tilde A)$ such that $\tilde A$ is $h$-thick and we have that:
    \begin{align*}
       \la_G( \tilde A) &\leq \la_G(A) + C\\
        \dfrac{||f||_p}{||\nabla^1_{2h} f||_p}&\leq C\dfrac{||\tilde f||_p}{||\nabla_h^1 \tilde f||_p}
    \end{align*}
\label{Thick}
\end{lem}


Let's consider $h$ large enough that the isoperimetric profile under $\nabla^1$ doesn't depend on $h$, and such that $S_G\subset B(1,h)$. We can then restrict ourselves to $h$-thick subsets, more precisely:
\begin{lem}
\label{thickprofile}
    Let's define:
    \begin{equation*}
    \tilde j_{h,G}(v)= \sup_{\la_G(A) \leq v, \textrm{A is } h-\textrm{thick}}  \overline J_{p,G}(A).
\end{equation*}
We then have that $\tilde j_{h,G} \approx j^r_{G}$.
\end{lem}
\begin{proof}
Let's denote by $j^1_h$ and $j^2_h$ the isoperimetric profiles with respect to $\nabla_h^1$ and $\nabla_h^2$ respectively.
It's clear that $\tilde j_{h,G} \leq j^r_{p,G}$, so in order to prove the converse inequality, consider $f$ in $L^{\infty}(A)$ with $\la_G(A)\leq v$. By the Lemma \ref{Thick} and since $||\nabla_h^1||_p \geq ||\nabla_h||_p \geq ||\nabla_{S_G}||_p$ we have that:
\begin{equation*}
    \dfrac{||f||_p}{||\nabla_{2h}^1 f||_p}\leq C\tilde j_{h,G}(v+C)
\end{equation*}
Which then implies that $j^1_{2h} \preccurlyeq \tilde j_{h,G}$. Since $j^1_{2h} \approx j^1_h \approx j^r_{p,G}$ we obtain the desired result.
\end{proof}


\subsection{Monotonicity of Isoperimetric Profile}
In all of this subsection, we consider $G,H$ being locally compact unimodular groups with fixed compact generating symmetric subsets $S_G,S_H$ respectively.

   \begin{lem}
    Consider $(\Om,\mu)$ a coarsely $m$-to-$1$, $L^p$-measure subgroup coupling from $H$ to $G$ with the same notation from the Definition \ref{coupling}. For any $f$ in $L^p(G,\la_G)$ with $\supp (f)$ of finite measure, let's consider $\tilde f$ in $L^p(\Om,\mu)$ defined by:
        \begin{equation*}
        \tilde f(g\ast x) = f(g)
        \end{equation*}
        for all $g$ in $G$ and $x$ in $\df_G$. We then have that $||\tilde f||_p^p= \nu_G(\df_G)\cdot||f||_p^p $ and that:
        \begin{equation*}
        ||\nabla_H \tilde f||^p_p \leq  C ||\nabla_G^r f||^p_p 
        \end{equation*}
        where $C=\sup_{s \in S_H}\int_{\df_G} |c(s,x)|_G^p d\nu_G(x) $ and $||\nabla_H \tilde f ||_p$ is the gradient with respect to the $H$-action on $L^p(\Om,\mu)$ coming from the action $H\act (\Om,\mu)$. 
    \label{Gradiente}
    \end{lem}
\begin{proof}
    Given any $s$ in $S_H$, we have that:
    \begin{align*}
        ||\tilde f- s^{-1}\ast\tilde f||_p^p &= \int_{\Om} |\tilde f(\omega)-\tilde f(s\ast \omega)|^p d\mu(\omega)\\
        &= \int_{G\times \df_G} |\tilde f(g\ast x)-\tilde f(s\ast g\ast x)|^p d\la_G(g) d\nu_G(x).
    \end{align*}
However since both actions commute, we have that $s\ast g\ast x = g\ast s\ast x$; this implies that $s\ast g\ast x= (g\cdot c(s,x))\ast (s\cdot x)$ with $g\ast c(s,x)$ in $G$ and $s\cdot x$ in $\df_G$ which in turn implies that:
\begin{equation*}
    \tilde f(s \ast g \ast x) =  f (g\cdot c(s,x)).
\end{equation*}
This then by Lemma \ref{ineq} implies that:
\begin{align*}
    ||\tilde f- s^{-1}\ast\tilde f||_p^p &= \int_{G\times \df_G} |f(g)- f(g\cdot c(s,x))|^p d\la_G(g) d\nu_G(x). \\
    &\leq \int_{\df_G} ||f-\rho(c(s,x))f||^p_p d\nu_G(x).\\
    &\leq \int_{\df_G} |c(s,x)|_G^p .||\nabla^r_G f||^p_p d\nu_G(x) \\
    &\leq C ||\nabla_G^r f||_p^p.   
\end{align*}
Which implies the desired result. 
\end{proof}


Moreover, we will also obtain the sub-linear version of the preceding lemma:

\begin{lem}
    Let's consider $\varphi\colon (0,\infty)\to (0,\infty)$ a function such that $\varphi$ and $t\mapsto \frac{t}{\varphi(t)}$ are non-decreasing functions. Consider $(\Om,\mu)$ a coarsely $m$-to-$1$, $\varphi$-integrable measure subgroup coupling from $H$  to $G$ with the same notations from Definition \ref{coupling}. For any $f$ in $L^1(G,\la_G)$ with $\supp(f)$ of finite measure, let's consider $\tilde f$ in $L^1(\Om,\mu)$ defined by:
    \begin{equation*}
        \tilde f(g*x)= f(g)
    \end{equation*}
    for all $g$ in $G$ and $x$ in $\mathcal F_G$. We then have that $||\tilde f||_1= \nu_G(\df_G)||f||_1$ and that:
    \begin{equation*}
        ||\nabla_H \tilde f||_1\leq C \frac{||f||_1}{\varphi(||f||_1)}
    \end{equation*}
    where $C=2\sup_{s\in S_G} \int_{\df_G} \varphi(|c(s,x)|_G ||\nabla_G^r f||_1) d\nu_G(x)$ and $||\nabla_H \tilde f ||_1$ is the gradient with respect to the $H$-action on $L^1(\Om,\mu)$ coming from the action $H\act (\Om,\mu)$.
\label{GradientSublinear}

\end{lem}
\begin{proof}
    Given any $s$ in $S_H$, we have that:
    \begin{align*}
        ||\tilde f- s^{-1}\ast\tilde f||_1 &= \int_{\Om} |\tilde f(\omega)-\tilde f(s\ast \omega)|d\mu(\omega)\\
        &= \int_{G\times \df_G} |\tilde f(g\ast x)-\tilde f(s\ast g\ast x)| d\la_G(g) d\nu_G(x).
    \end{align*}
However since both actions commute, we have that $s\ast g\ast x = g\ast s\ast x$; this implies that $s\ast g\ast x= (g\cdot c(s,x))\ast (s\cdot x)$ with $g\ast c(s,x)$ in $G$ and $s\cdot x$ in $\df_G$ which in turn implies that:
\begin{equation*}
    \tilde f(s \ast g \ast x) =  f (g\cdot c(s,x)).
\end{equation*}
This then by Lemma \ref{ineq2} implies that:
\begin{align*}
    ||\tilde f- s^{-1}\ast\tilde f||_1 &= \int_{G\times \df_G} |f(g)- f(g\cdot c(s,x))| d\la_G(g) d\nu_G(x). \\
    &\leq \int_{\df_G} ||f-\rho(c(s,x))f||_1  d\nu_G(x).  \\
    &\leq \int_{\df_G} 2\varphi(|c(s,x)|_G||\nabla_G^r f||_1)\frac{||f||_1}{\varphi(||f||_1)} d\nu_G(x). \\
    &\leq C \frac{||f||_1}{\varphi(||f||_1)}.  
\end{align*}
where $\rho\colon G\act L^p(G)$ is the right regular representation. We have thus proved the desired result. 
\end{proof}

We have proved in Lemma \ref{thickprofile}, that in order to compute the isoperimetric profile, we only need to restrict ourselves to $h$-thick subsets. This will be important in our proof of the monotonicity of the isoperimetric profile, as such we will need the following lemma:

\begin{lem}
\label{mto1}
Let $f\colon H\to G$ be an application such that the preimage of any ball in $G$ of radius $4h$ can be covered by at most $m$ balls of radius $s$ in $H$. If $A$ is $h$-thick, that is, it is the union of balls of radius at least $h$, and of finite measure; we then have that for some $C_1>0$:
\begin{equation*}
    \la_H(f^{-1}(A))\leq C_1.\la_G(A)
\end{equation*}
\end{lem}
\begin{proof}
    Since $A$ is $h$-thick we can suppose that there exists $C$ such that $A=[C]_h$. Let's consider $Z$ a maximal $3h$-discrete subset of $C$, by the lemma \ref{skeleton}, we have that $C\subset [Z]_{3h}$, which then implies $A\subset [Z]_{4h}$. Moreover, let's note that since $Z$ is $3h$-discrete, we also have that $[Z]_{h}\subset A$ which implies that:
    \begin{equation*}
        \#Z\cdot \la_G(B(1_G,h)) \leq \la_G(A)
    \end{equation*}
    By the hypothesis on $f$ we have that $\la_H(f^{-1}(A))\leq m\#Z \cdot \la_H(B(1_H,s))$ which implies that for $C_1=m\cdot \la_H(B(1_H,s))/ \la_G(B(1_G,h)) $, we have that:
    \begin{equation*}
        \la_H(f^{-1}(A))\leq C_1.\la_G(A)
    \end{equation*}
    Which concludes the proof.
\end{proof}

Given $(\Om,\mu)$ a coarsely $m$-to-$1$ measure subgroup coupling from $H$ to $G$ with all the notations from \ref{coupling}. Let's consider then the family of functions $(\tilde f_x)_{x\in \df_G}\colon H\to \mathbb R$ given by $\tilde f_x(h)=\tilde f(h*x)$, note that since $(\Om,\mu)$ is a coarsely $m$-to-$1$ coupling if $\supp f\subset A$ and $A$ is $h$-thick, by the lemma \ref{mto1} we have that there exists $C_1>0$ such that $\la_H(c_x^{-1}(A))\leq C_1\la_G(A)$. Since $\supp \tilde f_x\subset (c_x^{-1}(A))^{-1}$ and $H$ is unimodular, we then have that for all $x$ in $\df_G$:
\begin{equation*}
 \la_H(\supp \tilde f_x)\leq C_1\cdot\la_G(A)   
\end{equation*}

        

We will then have the following two theorems:

\begin{teo}
\label{teoLp}
    Let $H$ be a coarsely $m$-to-$1$ $L^p$-measure subgroup of $G$, we then have that:
    \begin{equation*}
        j_{p,G} \preccurlyeq j_{p,H}
    \end{equation*}
\end{teo}
\begin{proof}

Let's consider $\eps>0$ and $A$ a $h$-thick subset of $G$ of finite measure, there exists then $f$ in $L^p(G,\mu_G)$ such that $\supp(f) \subset A$ and:
\begin{equation*}
    J_{p,G}^r(A) -\epsilon \leq \dfrac{||f||_p}{||\nabla_G^r f||_p} \leq J_{p,G}^r(A)
\end{equation*}
Using Lemma \ref{Gradiente} we have that, there exists a constant $\overline C$ such that:
\begin{equation}
   \dfrac{||\nabla_H \tilde f||_p}{||\tilde f||_p}  \leq \overline C \dfrac{||\nabla_G^r f||_p}{||f||_p}. 
    \label{eq1}   
\end{equation}
Furthermore, this implies that there exists $Z\subset \df_H$ with $\nu_H(Z)>0$ such that:
\begin{equation*}
    ||\nabla_H^l \tilde f_y||_p <  2\overline C \dfrac{||\nabla_G^r f ||_p}{||f||_p}||\tilde f_y||_p  
\end{equation*}
for all $y$ in $Z$. This is true, since if that were not the case, we would have that for all $y$ in conull subset of $\df_H$:
\begin{equation*}
    ||\nabla_H^l \tilde f_y||_p \geq 2\overline C \dfrac{||\nabla_G^r f ||_p}{||f||_p}||\tilde f_y||_p  
\end{equation*}
Which by integrating over $\df_H$ would imply that:
\begin{equation*}
    ||\nabla_H \tilde f||\geq 2 \overline C \dfrac{||\nabla_G^r f ||_p}{||f||_p}||\tilde f||_p
\end{equation*}
that contradicts the inequality (\ref{eq1}).
For any $y$ in $Z$, we have that $\tilde f_y$ is nonzero. Moreover, since $(\Om,\mu)$ is a coarsely $m$-to-$1$ coupling, the Lemma \ref{mto1} implies that $\supp \tilde f_y$ has finite measure, which implies that $||\nabla^l_H \tilde f_y||_p>0$.
This implies that for any $y$ in $Z$ we have that:
\begin{equation*}
    \dfrac{||f||_p}{||\nabla^r_G f||_p} \leq 2\overline C \dfrac{||\tilde f_y||_p}{||\nabla_H^l \tilde f_y||_p}
\end{equation*}
Furthermore, we get:
\begin{equation*}
    J_p^r(A)-\epsilon \leq 2\overline C J^l_p(C_1v).
\end{equation*}
Since we can make $\epsilon$ converge to $0$, we obtain:
\begin{equation*}
    J_p^r(A) \leq 2\overline C J^l_p(C_1v).
\end{equation*}
Which then implies that:
\begin{equation*}
     j_{p,G}^r \preccurlyeq j^l_{p,H}
\end{equation*}
as we wanted.
\end{proof}
In a similar manner, we obtain the sublinear version of the last theorem:
\begin{teo}
\label{teosublinear}
Let's consider $\varphi\colon (0,\infty)\to (0,\infty)$ a function such that $\varphi$ and $t\to \frac{t}{\varphi(t)}$ are non-decreasing functions. If $H$ is a coarsely $m$-to-$1$ $\varphi$-integrable measure subgroup of $G$. We then have that:
    \begin{equation*}
        \varphi \circ j_{1,G}\preccurlyeq j_{1,H}
    \end{equation*}
    
\end{teo}
\begin{proof}
    Let's consider $\epsilon > 0$ and $A$ a $h$-thick subset of $G$ of finite measure, by definition, there exists $f$ in $L^1(G,\mu_G)$ such that $\supp f\subset A$ and:
    \begin{equation*}
        J^r_{1,G}(A)-\epsilon \leq  \frac{||f||_1}{||\nabla_G^r f||_1}\leq J^r_{1,G}(A)
    \end{equation*}
    Moreover, up to multiplying with a constant, we can choose $f$ such that $||\nabla_G^r f||_1=1$.
    Using Lemma \ref{GradientSublinear}, we obtain that:
    \begin{equation*}
         ||\nabla_H \tilde f||_1 \leq 2\sup_{s\in S_G}\int_{\df_G} \varphi(|c(s,x)|_G) d\nu_G(x) \frac{||f||_1}{\varphi(||f||_1)}.
    \end{equation*}
    Let's recall as well that $||\tilde f||_1=\nu_G(\df_G)||f||_1$ and that:
    \begin{equation*}
        ||\tilde f||_1=\int_{\df_H} ||\tilde f_y||_1 d\nu_H(y)  
    \end{equation*}
    If we denote $C=\dfrac{2\sup_{s\in S_G} \int_{\df_G} \varphi(|c(s,x)|_G)}{\nu_G(\df_G)}$, it is clear that we have that:
    \begin{equation}
        ||\nabla_H \tilde f||_1\leq \dfrac{C}{\varphi(||f||_1)} ||\tilde f||_1 \label{sublineareqn}
    \end{equation}
    We then obtain the following:
    \begin{claim}
        There exists a set $Z\subset \df_H$ with $\nu_H(Z)>0$ such that for all $y\in Z$, we have that:
        \begin{equation*}
            ||\nabla_H \tilde f_y||_1 < 2C \dfrac{||\tilde f_y||_1}{\varphi(||f||_1)}
        \end{equation*}
    \end{claim}
    \begin{proof}[Proof of Claim]
        If that were not the case, we would have that there exists a $\nu_H$-conull subset of $\df_H$ such that for all $y$ in this subset, we would have that:
        \begin{equation*}
           ||\nabla_H \tilde f_y||_1\geq \dfrac{2C}{\varphi(||f||_1)}||\tilde f_y||_1
        \end{equation*}
        which integrating over $y$ would gives us that:
        \begin{equation*}
            ||\nabla_H \tilde f||\geq \dfrac{2C}{\varphi(||f||_1)} ||\tilde f||_1
        \end{equation*}
        which is a contradiction with the inequality (\ref{sublineareqn}).
    \end{proof}
Note that for any $y$ in $Z$, we have that $||\tilde f_y||_1$ is nonzero. Moreover, since $(\Om,\mu)$ is a coarsely $m$-to-$1$ coupling, the Lemma \ref{mto1} implies that $\supp \tilde f_y$ has finite measure, which implies that $||\nabla^l_H \tilde f_y||_p>0$. Furthermore, for any $y$ in $Z$, we have that:
\begin{equation*}
    \varphi\left(\dfrac{||f||_1}{||\nabla^r_G f||_1}\right) < 2C\dfrac{||\tilde f_y||_1}{||\nabla_H \tilde f_y||_1}
\end{equation*}
Which, since $\varphi$ is increasing, implies that:
\begin{equation*}
    \varphi(J^r_{1,G}(A)-\epsilon) < 2CJ^l_{1,H}(C_1v)
\end{equation*}
Since we can make $\epsilon$ converge to $0$, we obtain:
\begin{equation*}
    J^r_{1,G}(A) \leq 2CJ^l(C_1v)
\end{equation*}
Which then implies that:
\begin{equation*}
     \varphi(j_{1,G}^r) \preccurlyeq j^l_{1,H}
\end{equation*}
as we wanted.
\end{proof}

\ \newline






\bibliographystyle{amsalpha}
\bibliography{refs.bib}

\end{document}